\documentclass[12pt,a4paper]{article}

\usepackage{amssymb}
\usepackage{amsmath, amsthm}
\usepackage{arydshln}
\usepackage{epsfig}
\usepackage{setspace}
\usepackage{comment}
\usepackage{geometry}
\usepackage{mathtools}
\usepackage[dvipdfmx]{}
\usepackage[dvipdfmx]{color}
\usepackage{algorithm, algpseudocode}

\newtheorem{defi}{Definition}[section]
\newtheorem{theo}[defi]{Theorem}
\newtheorem{lemm}[defi]{Lemma}

\newtheorem{coro}[defi]{Corollary}
\newtheorem{rem}{Remark}

\algtext*{EndWhile}
\algtext*{EndIf}
\algtext*{EndFor}
\algtext*{EndFunction}

\title{On integer network synthesis problem\\ with tree-metric cost}
\author{Hiroshi HIRAI and Masashi NITTA \\
Graduate School of Information Science and Technology,   \\
The University of Tokyo, Tokyo, 113-8656, Japan.\\
\texttt{\normalsize hirai@mist.i.u-tokyo.ac.jp}\\
\texttt{\normalsize masashi\_nitta@ipc.i.u-tokyo.ac.jp}
}
\begin{document}
\maketitle
\begin{abstract}
	Network synthesis problem (NSP) is the problem of designing a
	minimum-cost network (from the empty network) satisfying a given
	connectivity requirement.
	Hau, Hirai, and Tsuchimura showed that
	if the edge-cost is a tree metric,
	then a simple greedy-type algorithm solves NSP to obtain a
	half-integral optimal solution.
	This is a generalization of the classical result
	by Gomory and Hu for the uniform edge-cost.
	
	In this note, we present an integer version of Hau, Hirai, and
	Tsuchimura's result for integer network synthesis problem (INSP),
	where a required network must have an integer capacity.
	We prove that if each connectivity requirement is at least 2 and the
	edge-cost is a tree metrc, then INSP can be solved in polynomial time.
\end{abstract}


\section{Introduction}

{\it Network synthesis problem(NSP)} is one of the most simplest network design problems, which asks to find a minimum-cost network (from the empty network) satisfying a given connectivity requirement.
In a classic paper~\cite{ref:GomoryHu} of combinatorial optimization,
Gomory and Hu algorithmically proved that NSP admits a half-integral optimal solution provided the edge-cost is uniform.
Note that this half-integrality property fails under the general edge-cost.
Recently, Hau, Hirai, and Tsuchimura~\cite{ref:HHT} showed that this classical result is naturally generalized for a tree-metric
edge-cost.
Here a {\it tree metric} is a special metric represented by the
shortest path distance on a weighted tree.

In this note, we continue this line of research, and address the {\it integer network synthesis problem (INSP)},  where a
required network must have an integer capacity.
Since INSP is NP-hard in general, it is interesting  to explore a
class of edge-costs tractable for INSP.
An old result, due to Chou and Frank~\cite{ref:chou-frank}, says that
INSP can also be solved in polynomial time provided the edge-cost is uniform.
Frank~\cite{ref:Frank} developed an elegant framework, based on splitting-off, for connectivity augmentation problem (which generalizes NSP).
As a consequence of his result,
INSP can be solved in polynomial time provided the edge-cost is {\em node-induced}.
Here a node-induced cost is a special tree metric corresponding to a star.

The main result of this note is a partial generalization: If each
connectivity requirement is at least 2 and the edge-cost is a
tree metric, then INSP can be solved in polynomial time.
Our algorithm modifies Frank's framework to apply splitting-off to the tree-network corresponding to the tree-metric, whereas Frank applied splitting-off to the star-network (corresponding to a node-induced cost).

The rest of this paper is organized as follows.
In Section 2, we state our result.
In Section 3,  we describe our algorithm to solve INSP with tree-metric cost.
\paragraph{Notation}
Let $\mathbb{R}$ and $\mathbb{R}_+$ be the sets of reals and nonnegative reals, respectively.
Also, let $\mathbb{Z}$ and $\mathbb{Z}_+$ be the sets of integers and nonnegative integers, respectively.
For an undirected graph $G$ and a node subset $X \subseteq V(G)$, 
let $\delta_G(X)$ denote the set of edges connecting $X$ and $V(G)\setminus X$.
If the graph $G$ is obvious in the context, then $\delta_G(X)$ is written as $\delta(X)$.
Also $\delta(\{v\})$ is written as $\delta(v)$.
The {\it edge-connectivity} between $u$ and $v$ in $G$ is the maximum number of pairwise edge-disjoint paths connecting $u$ and $v$.
For a function $x \colon S \rightarrow \mathbb{R}$ on a set $S$ and a subset $X \subseteq S$, we denote $\sum_{u \in X} x(u)$ simply by $x(X)$.

\section{Network synthesis problem}
Let $K_V$ denote a complete undirected graph on node set $V$.
Let $r \colon V \times V \rightarrow \mathbb{Z}_+$ be a connectivity requirement, and let $a \colon E(K_V) \rightarrow \mathbb{R}_+$ be an edge-cost.
A {\it realization} (of $r$) is a capacity function $y \colon E(K_V) \rightarrow \mathbb{R}_+$ such that for all distinct $s, t \in V$, it holds $y(\delta(X)) \geq r(s, t)$ for any subset $X \subseteq V$, i.e., there is an $(s, t)$-flow of value $r(s, t)$ under capacity $y$.
The {\it cost} of a realization $y$ is defined as $\sum_{e \in E(K_V)} a(e)y(e)$.
The {\it network synthesis problem (NSP)} asks to find a realization of minimum-cost.

Gomory and Hu \cite{ref:GomoryHu} algorithmically proved that NSP admits a half-integral minimum-cost realization under a uniform edge-cost.
Let $R(X) \coloneqq \max\{r_{ij}\mid i\in X \not\ni j\}\ (\phi \not=X \subset V)$, where $R(\{u\})$ is simply written as $R(u)$.
\begin{theo}[{\cite{ref:GomoryHu}}]\label{theo:gomory-hu}
	Suppose that $a(e) = 1$ for $e \in E(K_V)$.
	Then the following hold:
	\begin{itemize}
		\item[{\rm(1)}]
		A half-integral minimum-cost realization exists, and can be found in polynomial time.
		\item[{\rm(2)}] 
		The minimum-cost is equal to 
		$\displaystyle \frac{1}{2}\sum_{u\in V}R(u).$
	\end{itemize}
\end{theo}
Recently, Hau, Hirai and Tsuchimura \cite{ref:HHT} generalized Theorem \ref{theo:gomory-hu} to a tree-metric cost.
Here an edge-cost $a \colon E(K_V) \rightarrow \mathbb{R}_+$ is called a {\it tree-metric} if there exist a tree $T$ with $V \subseteq V(T)$ and a nonnegative edge-length $l\colon E(T) \rightarrow \mathbb{R}_+$ such that for all distinct $i, j \in V$, $a(ij)$ is equal to the length of the unique path between $i$ and $j$ in $T$ with respect to $l$.
In this case, we say that $a$ is {\it represented} by $T$ and $l$.
Fix an arbitrary $v \in V$.
For edge $e \in E(T)$, let $X_e \subseteq V$ denote the subset of $V$ consisting of nodes reachable from $v$ in $T$ not using edge $e$.
\begin{theo}[\cite{ref:HHT}]\label{theo:HHT}
	Suppose that $a$ is a tree metric represented by a tree $T$  and a nonnegative edge-length $l\colon E(T) \rightarrow \mathbb{R}_+$.
	Then the following hold:
	\begin{itemize}
		\item[{\rm(1)}] 
		A half-integral minimum-cost realization exists, and can be found in polynomial time.
		\item[{\rm(2)}] 
		The minimum-cost is equal to $\displaystyle \sum_{e \in E(T)} l(e)R(X_e).$
	\end{itemize}
\end{theo}
This theorem is a generalization of Theorem~\ref{theo:gomory-hu}.
Indeed, the uniform edge-cost is represented by $T$ and $l$ as:
\begin{itemize}
	\item $V(T) \coloneqq V \cup \{s\}$ and $E(T) \coloneqq \{sv \mid v \in V\}$.
	\item $l(e) \coloneqq 1/2$ for $e \in E(T)$.
\end{itemize}

Next, we consider an integer version {\it INSP} of NSP, where a realization $y$ must satisfy $y(e) \in \mathbb{Z}_+$ for each edge $e \in E(K_V)$.
Chou and Frank~\cite{ref:chou-frank} gave a polynomial time algorithm for finding a minimum-cost integer realization together with a formula of the minimum-cost:
\begin{theo}[\cite{ref:chou-frank}]\label{theo:frank-chou}
	Suppose that $a(e) = 1$ for $e \in E(K_V)$.
	Then the following hold:
	\begin{itemize}
		\item[{\rm (1)}] 
		A minimum-cost integer realization can be found in polynomial time.
		\item[{\rm (2)}] 
		If $R(u) \not= 1$ for all $u \in V$, then the minimum-cost is equal to $\displaystyle \left\lceil \frac{1}{2}\sum_{u\in V}R(u) \right\rceil.$
	\end{itemize}
\end{theo}
In the case where there is a node $u$ with $R(u) = 1$, the problem easily reduces to that on $K_{V\setminus\{u\}}$ (by choosing one edge connecting $u$)~\cite{ref:Frank, ref:FrankText, schrijver}.

The main result of this note is a tree-metric version of Theorem \ref{theo:frank-chou}.
Suppose that $a$ is a tree metric represented by a tree $T$ and a nonnegative edge-length $l\colon E(T) \rightarrow \mathbb{R}_+$.
From connectivity requirement $R$, define an edge-capacity $c^R\colon E(T) \rightarrow \mathbb{Z}_+$ on $T$ by
\begin{align*}
	c^R(e) \coloneqq R(X_e)\quad(e \in E(T)).
\end{align*}
An {\it inner-odd join} (with respect to ($T, R$)) is an edge subset $F \subseteq E(T)$ such that
\begin{align*}
	c^R(\delta_T(v)) \equiv |\delta_T(v) \cap F|\ \mbox{mod 2}\quad(v \in V(T)\setminus V).
\end{align*}
\begin{theo}
\label{theo:ore}
	Suppose that $a$ is a tree metric represented by a tree $T$ and a nonnegative edge-length $l\colon E(T) \rightarrow \mathbb{R}_+$.
	Also, suppose that $R(X_e) > 1$ for all $e \in E(T)$. 
	Then the following hold:
	\begin{itemize}
		\item[{\rm (1)}] 
		A minimum-cost integer realization can be found in polynomial time.
		\item[{\rm (2)}] 
		The minimum-cost is equal to 
		\begin{align*}
			\sum_{e \in E(T)} l(e)R(X_e) + \min_{F}\sum_{e\in F}l(e),
		\end{align*}
		where the minimum is taken over all inner odd joins $F$ with respect to $(T, R)$.
	\end{itemize}
\end{theo}
This generalizes Theorem \ref{theo:frank-chou} in the case of $R(u) > 1$.
Indeed, in the above star representation of the uniform cost, a minimum cost inner-odd join is $\{sv\}$ if $\sum_{u\in V} R(u)$ is odd, and $\emptyset$ otherwise.
We do not know whether tree-metric weighted INSP including $R(X) \leq 1$
is solvable in polynomial time; 
see also Remark~\ref{rem1}.
The proof of Theorem \ref{theo:ore} is given in the next section.

\begin{rem}
\label{rem1}
{\rm 
In the case where $R(X_e) = 0$ for some $e \in E(T)$, one
may naturally expect that  the problem is decomposed to that on
$K_{X_e}$ and on $K_{V \setminus X_{e}}$.
This is not true for INSP (but is true for NSP).
Indeed, let $V \coloneqq \{u_1,u_2,u_3,v_1,v_2,v_3\}$.
The requirement $r$ is defined by $r(u_i, u_j) = r(v_i, v_j) \coloneqq 3$ for $1 \leq i < j \leq 3$ and $r(u_i, v_j) \coloneqq 0$ for $1 \leq i,j \leq 3$.
Let $a$ be a tree metric on $V$ represented by $T,l$,
where $V(T) := V \cup \{u,v\}$, $E(T) := \{uv,
uu_1,uu_2,uu_3,vv_1,vv_2,vv_3\}$, $l(uv) = 1$, and $l(uu_i) = l(vv_i)
:= 2$ for $i=1,2,3$.
Now $R(X_{uv}) = 0$.
However, the problem cannot be decomposed; a minimum-cost integer
realization $y$ is given by 
$y(u_iu_j) = y(v_iv_j) = 1$ for $1 \leq i < j \leq 3$
and $y(u_iv_j) = 1$ if $i=j$, and 0 if $i \neq j$.
}
\end{rem}

\begin{rem}
\label{rem2}
{\rm
Suppose that we are further given capacity lower bound $g \colon E(K_V) \rightarrow \mathbb{Z}_+$.
The {\it edge-connectivity augmentation problem (ECAP)} is to find a minimum-cost increase $z \colon E(K_V) \rightarrow \mathbb{R}_+$ such that $g+z$ is a realization of $r$.
Frank~\cite{ref:Frank} shows that ECAP admits a half-integral optimal solution provided that the edge-cost is node-induced.
From this result and Theorem \ref{theo:HHT}, one can naturally conjecture that ECAP with a tree-metric cost has a half-integral optimal solution.
But this conjecture is wrong.
Indeed, we found instances having optimal values 638+1/3 and 9572+1/4~\cite{ref:nitta}.
}
\end{rem}

\section{Algorithm}
Let $V, r, a$ be an instance of INSP as above.
Suppose that $a$ is a tree metric represented by $T$ and $l$,
where $X_e$ is defined as above for $e \in E(T)$.
An integer edge-capacity $c \colon E(T) \rightarrow \mathbb{Z}_+$ is said to be $r$-{\it feasible} if $c$ satisfies the following:
\begin{itemize}
	\item[(c1)] 
	 $c(\delta_T(v))$ is even for all $v \in V(T)\setminus V$.
	 \item[(c2)] 
	 $c(e) \geq R(X_e)$ holds for all $e \in E(T)$.
\end{itemize}
The cost of an $r$-feasible edge-capacity $c$ is defined as $\sum_{e \in E(T)} l(e)c(e)$.
The essence of Theorem \ref{theo:ore} is the following:
\begin{theo}\label{theo:main1}
	Suppose that $R(X_e) > 1$ for all $e \in E(T)$.
	Then the following hold:
	\begin{itemize}
		\item[{\rm (1)}]
		For an integral realization $y$ of $r$, 
		define an edge-capacity $c \colon E(T) \rightarrow \mathbb{Z}_+$ by:
		\begin{align*}\label{defEdgeWeight}
		c(e) \coloneqq \sum \{y(ij) \mid ij \in E(K_V),\{i,j\}\cap X_e = 1\} \quad (e \in E(T)).
		\end{align*}
		Then $c$ is $r$-feasible with $\sum_{e \in E(T)}l(e)c(e) = \sum_{e \in E(K_V)}a(e)y(e)$.
		\item[{\rm (2)}] 
		% in polynomial time は入れたくなかった
		For an $r$-feasible edge-capacity $c \colon E(T) \rightarrow  \mathbb{Z}_+$, there is an integral realization~$y$ of $r$ with $\sum_{e \in E(K_V)}a(e)y(e) \leq \sum_{e \in E(T)}l(e)c(e)$.
	\end{itemize}
\end{theo}
\begin{coro}\label{coro:coro}
	The minimum cost of an integer realization of $r$ is equal to the minimum cost of an $r$-feasible capacity in $T$.
\end{coro}
From this corollary, we easily obtain Theorem \ref{theo:ore} (2) as follows:
Observe that a minimum-cost $r$-feasible edge-capacity is obtained from $c^R$ by increasing $c^R(e)$ by 1 on edge $e$ in an inner-odd join $F$,
where the cost is equal to $\sum_{e\in E(T)} l(e) c^R(e) + l(F) = \sum_{e\in E(T)} l(e) R(X_e) + l(F)$.
Thus we obtain the formula in Theorem 2.4 (2).

We first show (1) in Theorem \ref{theo:main1}.
\begin{proof}[Proof of Theorem \ref{coro:coro} (1)]
	(c1) Consider $ij \in E(K_V)$ with $y(ij) > 0$.
	Then $y(ij)$ contributes to $c$ along the unique path $P$ in $T$ connecting $i$ and $j$.
	Thus $y(ij)$ contributes to $c(\delta(v))$ by $2y(ij)$ if $v\in V(P)\setminus V$ and 0 if $v \not\in V(P)$.
	
	(c2) If there exists an edge $e$ such that $c(e) < R(X_e)$, then the cut size $y(\delta(X_e))$ is less than $R(X_e)$ in graph $K_V$ under the capacity $y$. 
	This contradicts the assumption that $y$ is a realization of $r$.
\end{proof}
Next we show (2).
Let $c \colon E(T) \rightarrow \mathbb{Z}_+$ be an $r$-feasible edge-capacity.
\begin{lemm}
	Let $G$ be an undirected graph obtained from $T$ by replacing every edge $e \in E(T)$ by $c(e)$ parallel edges.
	For $u, v \in V$, the edge-connectivity between $u$ and $v$ in $G$ is at least $r(u, v)$.
	%The maximum-flow value from $u$ to $v$ is at least $r(u, v)$ in $T$ under the capacity $c$ with (b2).
\end{lemm}
\begin{proof}
	The edge-connectivity between $u$ and $v$ is the minimum of $c(e)$ over edges $e$ in the unique path in $T$ connecting $u$ and $v$.
	For all $e$ in the path, it holds that $r(u, v) \leq R(X_e) \leq c(e)$, where
	the last inequality follows from (c2).
	Therefore the edge-connectivity between $u$ and $v$ is at least $r(u, v)$.
\end{proof}
In order to construct an integral realization of $r$ from edge-capacity $c$, we utilize Mader's theorem~\cite{ref:mader}, which is a basis of splitting-off technique. 
The {\it splitting off} a pair of edges $uv, vz$ is to replace the two edges $uv$, $vz$ by a new edge $e = uz$. 
Note that $e$ is a loop if $u=z$.
Call a pair of edges $vs, st$ {\it splittable} if they can be split off so as to preserve the edge-connectivity between every two nodes other than $s$.
\begin{theo}[\cite{ref:mader}]\label{theo:mader}
	Let $G = (V, E)$ be a connected undirected graph and let $s$ be a node of $V$ with $|\delta_G(s)| \not\in \{1, 3\}$.
	If there is no cut-edge incident to $s$, then there exists a splittable pair $(su, sv)$ of edges.
\end{theo}
We are ready to prove Theorem \ref{theo:main1}~(2).
For each node $s \in V(T)\setminus V$, We repeatedly use splitting off so as to preserve edge-connectivity between every two nodes other than $s$.
More precisely, we execute the Algorithm~\ref{algorithm1}.
\begin{algorithm}
\caption{Obtain an integer realization $y$ from $r$-feasible edge-capacity}
\label{algorithm1}
\begin{algorithmic}[1]
\Require A tree $T$ with $V \subseteq V(T)$ and an $r$-feasible edge-capacity $c \colon E(T)\rightarrow\mathbb{Z}_+$. 
\Ensure An integer realization $y$ of $r$.
\Function{ObtainIntegerRealization}{$T, c$}
\State Construct a graph $G$ from $T$ by replacing every edge $e \in E(T)$ by $c(e)$ parallel edges
\While{there exists a node $s \in V(G)\setminus V$ with $|\delta_{G}(s)| \not= 0$}
\State Repeat splitting off at $s$ so as to preserve the edge-connectivity other than $s$ unless $|\delta_{G}(s)| = 0$
\EndWhile
\For{$uv \in E(K_V)$}
\State $y(uv) \leftarrow$ the number of edges between $u$ and $v$ in $G$
\EndFor
\State\Return $y$
\EndFunction
\end{algorithmic}
\end{algorithm}
It is clear that $|\delta_G(s)|$ is always even in Line 4.
The following lemma and Theorem \ref{theo:mader} guarantee that there always exists a splittable pair of edges incident to $s$ in Line 4.
\begin{lemm}\label{lemm:construct}
	In Line 4, there is no cut-edge incident to $s$.
\end{lemm}
\begin{proof}
	In the initial $G$, the edge-connectivity between every pair of vertices is at least 2 since $c(e) \geq R(X_e) > 1$ (by (c2)).
	If $s$ has only one neighbor $v$, then there are at least two edges between $s$ and $v$;
	the claim of this lemma is obvious.
	If $s$ has at least two neighbors and the cut-edge $e$, the graph is divided into two components $U_1$ and $U_2$ by the deletion of $e$.
	Since $s$ has at least two neighbors, we can choose $u_1 \in U_1\setminus\{s\}$ and $u_2 \in U_2\setminus\{s\}$.
	Since the edge~$e$ is a cut-edge, the edge-connectivity between $u_1$ and $u_2$ is 1.
	On the other hand, the edge-connectivity between $u_1$ and $u_2$ is not changed in the middle of the algorithm.
	This is a contradiction.
\end{proof}
The following lemma shows that the cost of realization of $r$ is at most $\sum_{e \in E(T)} l(e)c(e)$ and we have proved Theorem \ref{theo:main1} (2).
\begin{lemm}
	Suppose that $y$ is the output of Algorithm 1.
	Then $\sum_{e \in E(K_V)}a(e)y(e) \leq \sum_{e \in E(T)}l(e)c(e)$ holds.
\end{lemm}
\begin{proof}
	Algorithm 1 repeats splitting-off operations.
	For each step, we define $z\colon V(G)\times V(G) \rightarrow \mathbb{Z}_+$ and $d\colon V(G)\times V(G)\rightarrow \mathbb{R}_+$ as follows:
	\begin{align*}
		z(u, v) &\coloneqq \mbox{the number of edges between $u$ and $v$}, \\
		d(u, v) &\coloneqq \mbox{the shortest length of a path in $T$ connecting $u$ and $v$ with respect to $l$}.
	\end{align*}
	Then the following hold:
	\begin{itemize}
		\item Before the initial splitting-off, $\sum_{u, v \in V(G)}d(u, v)z(u, v) = \sum_{e \in E(T)}l(e)c(e)$
		\item After the last splitting-off, $\sum_{u, v \in V(G)}d(u, v)z(u, v) = \sum_{e \in E(K_V)}a(e)y(e)$
	\end{itemize}
	Thus it suffices to show that $\sum_{u, v \in V(G)}d(u, v)z(u, v)$ does not increase for each splitting-off operation.
	Suppose that the splitting-off is applied to a pair $uv, vw$.
	Then the difference of $\sum_{u, v \in V(G)}d(u, v)z(u, v)$ is $d(u,w)-d(u,v)-d(v,w)$ and is nonpositive by the triangle inequality.
	Hence $\sum_{u, v \in V(G)}d(u, v)z(u, v)$ is non-increasing.
\end{proof}
\begin{algorithm}
	\caption{Find a minimum-cost integer realization}
	\label{algorithm2}
	\begin{algorithmic}[1]
		\Require A tree $T$ and a connectivity requirement $r$ (and $R$).
		\Ensure A minimum-cost integer realization $y$ of $r$.
		\Function{FindMinimumCostIntegerRealization}{$T, R$}
		\State Find a minimum cost inner-odd join $F$ with respect to $(T,R)$	(by Algorithm~\ref{algorithm3} below)
		\For{$e \in E(T)$}
			\State $c(e) \leftarrow c^R(e)$
			\If{$e\in F$}
			\State $c(e) \leftarrow c(e)+1$
			\EndIf
		\EndFor
		\State $y \leftarrow$ \Call{ObtainIntegerRealization}{$T$, $c$}
		\State\Return $y$.
		\EndFunction
	\end{algorithmic}
\end{algorithm}
\begin{algorithm}
	\caption{Find a minimum-cost $(I, J)$-join}
	\label{algorithm3}
	\begin{algorithmic}[1]
		\Require A tree $T$ and a pair $(I, J)$ of disjoint node subsets in $V(T)$.
		\Ensure A minimum-cost $(I, J)$-join $F$.
		\Function{FindMinimumCostJoin}{$T, I, J$}
		\If{$|V(T)| = 1$}
		\If{$J = \emptyset$}
		\State\Return $\emptyset$
		\Else
		\State\Return no solution
		\EndIf
		\EndIf
		\State Choose an arbitrary edge $e = uv \in E(T)$.
		\State $T^1, T^2 \leftarrow $ trees obtained from $T$ by the deletion of $e$.
		\For{$i = 1, 2$}
		\State $I^{i} \leftarrow V(T^{i}) \cap I, J^{i} \leftarrow V(T^{i}) \cap J$
		\State $I_{e}^{i} \coloneqq I^{i} \triangle \{u, v\}, J_{e}^{i} \coloneqq J^{i} \triangle\{u, v\}$
		\State $F^i \leftarrow $ \Call{FindMinimumCostJoin}{$T, I^i, J^i$}
		\State $F^i_e \leftarrow $ \Call{FindMinimumCostJoin}{$T, I^i_e, J^i_e$}
		\EndFor
		\If{$l(F^{1}) + l(F^{2}) < l(F^{1}_{e}) + l(F^{2}_{e}) + l(e)$}
		\State\Return $F^1 \cup F^2$.
		\Else
		\State\Return $F^{1}_{e}\cup F^{2}_{e} \cup \{e\}$.
		\EndIf
		\EndFunction
	\end{algorithmic}
\end{algorithm}
Algorithm~\ref{algorithm2} finds a minimum-cost integer realization.
To complete the proof of Theorem \ref{theo:ore}, we show that Algorithm 2 runs in polynomial time.
First we deal with Line 2 to 4 in Algorithm 2.
We consider the following generalized problem:
Given a tree $T$, disjoint node sets $I, J \subseteq V(T)$ and a nonnegative edge-cost $l \colon E(T) \rightarrow \mathbb{R}_+$, find a minimum-cost edge subset $F\subseteq E(T)$ satisfying:
\begin{align*}
	|\delta_T(v)\cap F| \equiv \begin{cases}&0 \quad\mbox{ mod 2}\quad\mbox{if } v \in I,\\
						&1 \quad\mbox{ mod 2}\quad\mbox{if } v \in J,
						\end{cases}
\end{align*}
where the cost of $F$ is defined as $l(F)$.
Such an edge subset is called an {\it $(I, J)$-join}.
Note that an inner-odd join is exactly an $(I, J)$-join for
\begin{align*}
	I &\coloneqq \{v \in V(T)\setminus V \mid c^R(\delta_T(v)) \mbox{ is even}\},\\
	J &\coloneqq (V(T)\setminus V) \setminus I.
\end{align*}
Algorithm~\ref{algorithm3} is a simple dynamic programming to find a minimum-cost $(I,J)$-join,
where $\triangle$ represents symmetric difference.
The correctness of this algorithm follows from the observation that $F^1 \cup F^2$ is a minimum-cost $(I, J)$-join not using edge $e$ and that $F^1_e \cup F^2_e \cup \{e\}$ is a minimum-cost $(I, J)$-join using edge $e$.

Finally we discuss the time complexity of the whole algorithm.
Obviously, Algorithm 3 is implemented by $O(n)$ time.
A naive implementation of Algorithm 2 does not yield a strongly polynomial time algorithm, since we replaced each $e \in E(T)$ by $c(e)$ parallel edges.
Therefore we regard $c(e)$ as a capacity on $e \in E(T)$ and we use a capacitated version of splitting-off.
It takes $O(n^6)$ time to complete splitting-off at node $s$; see~\cite{ref:Frank}.
There may be $O(n)$ nodes in $V(T)\setminus V$.
Thus the overall complexity of the whole algorithm is $O(n^7)$.
\section*{Acknowledgments}
The work was partially supported by JSPS KAKENHI Grant Numbers 25280004, 26330023, 26280004, 17K00029.

\end{document}